\documentclass[twoside,11pt]{article}

\usepackage{jmlr2e}
\usepackage{amsmath}
\usepackage[english,ruled,vlined]{algorithm2e}

\usepackage[T1]{fontenc}
\usepackage{setspace}

\ShortHeadings{}{}

\begin{document}

\title{\textsc{Multiplying a Gaussian Matrix by a Gaussian Vector}}

\author{\name Pierre-Alexandre Mattei \email \emph{pierre-alexandre.mattei@parisdescartes.fr} \\
       \addr Laboratoire MAP5, UMR CNRS 8145\\
       Universit\'e Paris Descartes \& Sorbonne Paris Cit\'e \\
       45 rue des Saints-P\`eres, 75006 Paris, France
}

\maketitle

\begin{abstract}
We provide a new and  simple characterization of the multivariate generalized Laplace distribution. In particular, this result implies that the product of a Gaussian matrix with independent and identically distributed columns by an independent isotropic Gaussian vector follows a \emph{symmetric} multivariate generalized Laplace distribution.
\end{abstract}

\section{Introduction}

\cite{wishart} proved that the inner product of two independent bidimensional standard gaussian vectors follows a Laplace distribution. This result is deeply linked to the fact that the Laplace distribution can be represented as an infinite scale mixture of Gaussians with gamma mixing distribution. Specifically, if $\sigma^2$ follows a $\textup{Gamma}(1,1/2)$ distribution and $x|\sigma \sim \mathcal{N}(0,\sigma^2)$, then $x$ follows a standard Laplace distribution\footnote{The shape-rate parametrization of the gamma distribution is used through this paper. Note also that a standard Laplace distribution is centered with variance two.}.
This representation -- which was recently named the \emph{Gauss-Laplace representation} by \cite{ding2017} following a blog post by Christian P. Robert\footnote{https://xianblog.wordpress.com/2015/10/14/gauss-to-laplace-transmutation/} -- is particularly useful if one wants to simulate a Laplace random variable: its use constitutes for example the cornerstone of the Gibbs sampling scheme for the Bayesian lasso of \cite{park2008}.

In this short paper, we are interested in studying links between multivariate counterparts of these two characterizations. More specifically, we give a new simple characterization of the \emph{multivariate generalized Laplace distribution} of \cite*{kotz2001}. As a corollary, we show that the product of a zero-mean Gaussian matrix with in dependent and identically distributed (i.i.d.) columns and a zero-mean isotropic Gaussian vector follows a symmetric multivariate generalized Laplace distribution, a result that has useful applications for Bayesian principal component analysis \citep*{bouveyron2016,bouveyron2017}.

In the remainder of this paper, $p$ and $d$ are two positive integers and $\mathcal{S}_p^+$ denotes the cone of positive semidefinite matrices of size $p\times p$.

\section{The multivariate generalized Laplace distribution}

While the definition of the univariate Laplace distribution is widely undisputed, there exist several different generalizations of this distribution to higher dimensions -- a comprehensive review of such generalizations can be found in the monograph of \citet*{kotz2001}. In particular, \cite{mcgraw1968} introduced a zero-mean elliptically contoured bidimensional Laplace distribution with univariate Laplace marginals. This distribution was later generalized to the $p$-dimensional setting by \cite{anderson1992}, considering characteristic functions of the form
$$\forall \mathbf{u} \in \mathbb{R}^p, \; \phi(\mathbf{u})=\frac{1}{1+\frac{1}{2} \mathbf{u}^T\mathbf{\Sigma}\mathbf{u}},
$$
where $\mathbf{\Sigma}\in \mathcal{S}_p^+$. This distribution was notably promoted by \cite*{eltoft2006} and is arguably the most popular multivariate generalization of the Laplace distribution \citep*[p. 229]{kotz2001}. Among its advantages, this distribution can be slightly generalized to model skewness, by building on characteristic functions of the form
$$\forall \mathbf{u} \in \mathbb{R}^p, \; \phi(\mathbf{u})=\frac{1}{1+\frac{1}{2} \mathbf{u}^T\mathbf{\Sigma}\mathbf{u}- i \boldsymbol{\mu}^T\mathbf{u}},
$$ where $\boldsymbol{\mu}\in \mathbb{R}^p$ accounts for asymmetry. Similarly to the univariate Laplace distribution, this asymmetric multivariate generalization is infinitely divisible \citep*[p. 256]{kotz2001}. Therefore, it can be associated with a specific Lévy process \cite[p. 5]{kyprianou2014}, whose increments will follow yet another generalization of the Laplace distribution, the  \emph{multivariate generalized asymmetric Laplace distribution}. This distribution, introduced by \citet*[p. 257]{kotz2001} and further studied by \citet*{kozubowski2013}, will be the cornerstone of our analysis of multivariate caracterizations of Laplace and Gaussian distributions.

\begin{definition} A random variable $\mathbf{z} \in \mathbb{R}^p$ is said to have a \textbf{multivariate generalized asymmetric Laplace distribution} with parameters $s>0, \boldsymbol{\mu} \in \mathbb{R}^p$ and $\mathbf{\Sigma} \in \mathcal{S}_p^+$ if its characteristic function is $$\forall \mathbf{u} \in \mathbb{R}^p, \; \phi_{ \textup{GAL}_p(\mathbf{\Sigma}, \boldsymbol{\mu},s)}(\mathbf{u})=\left(\frac{1}{1+\frac{1}{2} \mathbf{u}^T\mathbf{\Sigma}\mathbf{u} - i \boldsymbol{\mu}^T\mathbf{u}}\right)^s.$$
It is denoted by $\mathbf{z} \sim  \textup{GAL}_p(\mathbf{\Sigma}, \boldsymbol{\mu},s)$.
\end{definition}
General properties of the generalized asymmetric Laplace distribution distribution are discussed by \citet*{kozubowski2013}. We list here a few useful ones.
\begin{proposition} Let $s>0, \boldsymbol{\mu} \in \mathbb{R}^p$ and $\mathbf{\Sigma} \in \mathcal{S}_p^+$. If $\mathbf{z} \sim  \textup{GAL}_p(\mathbf{\Sigma}, \boldsymbol{\mu},s)$, we have $\mathbb{E}(\mathbf{z})=s\boldsymbol{\mu}$ and $\mathbb{C}\textup{ov}(\mathbf{z})=s(\mathbf{\Sigma} +\boldsymbol{\mu}\boldsymbol{\mu}^T )$. Moreover, if  $\mathbf{\Sigma}$ is positive definite, the density of $\mathbf{z}$ is given by $$
\forall \mathbf{x} \in \mathbb{R}^p, \;  f_\mathbf{z}(\mathbf{x}) = \frac{2 e^{\boldsymbol{\mu}^T\mathbf{\Sigma}^{-1}\mathbf{x}}}{(2\pi)^{p/2}\Gamma(s)\sqrt{\det{\boldsymbol{\Sigma}}}} \left( \frac{Q_{\mathbf{\Sigma}}(\mathbf{x})}{C(\boldsymbol{\Sigma},\boldsymbol{\mu})}\right)^{s-p/2} K_{s-p/2} \left(Q_{\mathbf{\Sigma}}(\mathbf{x}) C(\boldsymbol{\Sigma},\boldsymbol{\mu})\right), $$ where $ Q_{\mathbf{\Sigma}}(\mathbf{x})= \sqrt{\mathbf{x}^T\mathbf{\Sigma}^{-1}\mathbf{x}}$, $C(\boldsymbol{\Sigma},\boldsymbol{\mu})=\sqrt{2 + \boldsymbol{\mu}^T\mathbf{\Sigma}^{-1}\boldsymbol{\mu}}$ and $K_{s-p/2}$ is the modified Bessel function of the second kind of order ${s-p/2}$.
\end{proposition}
Note that the $ \textup{GAL}_1(2b^2,0,1)$ case corresponds to a centered univariate Laplace distribution with scale parameter $b>0$. In the symmetric case ($\boldsymbol{\mu}=0$) and when $s=1$, we recover the multivariate generalization of the Laplace distribution of \cite{anderson1992}.

An appealing property of  the multivariate generalized Laplace distribution is that it is also endowed with a multivariate counterpart of the Gauss-Laplace representation.

\begin{theorem}[Generalized Gauss-Laplace representation]
Let $s>0, \boldsymbol{\mu} \in \mathbb{R}^p$ and $\mathbf{\Sigma} \in \mathcal{S}_p^+$. If $u \sim \textup{Gamma}(s,1)$ and $\mathbf{x} \sim \mathcal{N}(0,\mathbf{\Sigma})$ is independent of $u$, we have\begin{equation} u \boldsymbol{\mu} + \sqrt{u}\mathbf{x}\sim  \textup{GAL}_p(\mathbf{\Sigma}, \boldsymbol{\mu},s). \label{normalmean}
\end{equation}
\end{theorem}
A proof of this result can be found in \citet*[chap. 6]{kotz2001}. This representation explains why the multivariate generalized Laplace distribution can also be seen as a multivariate generalization of the \emph{variance-gamma distribution} which is widely used in the field of quantitative finance \citep*{madan1998}. Infinite mixtures similar to \eqref{normalmean} are called \emph{variance-mean mixtures} \citep*{barndorff1982} and are discussed for example by \cite{yu2017}.

Another useful property of the multivariate generalized Laplace distribution is that, under some conditions, it is closed under convolution.

\begin{proposition} \label{summation}
Let $s_1,s_2>0, \boldsymbol{\mu} \in \mathbb{R}^p$ and $\mathbf{\Sigma} \in \mathcal{S}_p^+$. If $\mathbf{z}_1\sim  \textup{GAL}_p(\mathbf{\Sigma}, \boldsymbol{\mu},s_1)$ and $\mathbf{z}_2\sim  \textup{GAL}_p(\mathbf{\Sigma}, \boldsymbol{\mu},s_2)$ are independant random variables, then \begin{equation} \mathbf{z}_1+\mathbf{z}_2\sim  \textup{GAL}_p(\mathbf{\Sigma}, \boldsymbol{\mu},s_1+s_2).
\end{equation}
\end{proposition}
\begin{proof} Since $\mathbf{z}_1$ and $\mathbf{z}_2$ are independent, we have for all $\mathbf{u} \in \mathbb{R}^p$, $$\phi_{\mathbf{z}_1+\mathbf{z}_2}( \mathbf{u})=\phi_{ \textup{GAL}_p(\mathbf{\Sigma}, \boldsymbol{\mu},s_1)}( \mathbf{u})\phi_{ \textup{GAL}_p(\mathbf{\Sigma}, \boldsymbol{\mu},s_2)}( \mathbf{u})=\left(\frac{1}{1+\frac{1}{2} \mathbf{u}^T\mathbf{\Sigma}\mathbf{u} - i \boldsymbol{\mu}^T\mathbf{u}}\right)^{s_1+s_2}$$ which is the characteristic function of the $ \textup{GAL}_p(\mathbf{\Sigma}, \boldsymbol{\mu},s_1+s_2)$ distribution.
\end{proof}

\section{A new characterization involving a product between a Gaussian matrix and a Gaussian vector}

We now state our main theorem, which gives a new characterization of multivariate generalized Laplace distributions with half-integer shape parameters.

\begin{theorem} \label{maintheo}
Let $\mathbf{W}$ be a $p \times d$ random matrix with i.i.d. columns following a $\mathcal{N}(0,\mathbf{\Sigma})$ distribution, $\mathbf{y}\sim \mathcal{N}(0,\mathbf{I}_d)$ be a Gaussian vector independent from $\mathbf{W}$ and let $\boldsymbol{\mu} \in \mathbb{R}^p$. We have \begin{equation} \label{mainres}
\mathbf{Wy} + ||\mathbf{y}||_2^2\boldsymbol{\mu}\sim  \textup{GAL}_p(2\mathbf{\Sigma}, 2\boldsymbol{\mu},d/2).
\end{equation}

\end{theorem}
\begin{proof}
For each $k \in \{1,...,d\}$ let $\mathbf{w}_k$ be the $k$-th column of $\mathbf{W}$, $u_k=y_k^2$ and $\boldsymbol{\xi}_k=y_k \mathbf{w}_k +  y_k^2\boldsymbol{\mu}$. To prove the theorem, we will prove that $\boldsymbol{\xi}_1,...,\boldsymbol{\xi}_d$ follow a $ \textup{GAL}$ distribution and use the decomposition $$\mathbf{Wy} + ||\mathbf{y}||_2^2 \boldsymbol{\mu}=\sum_{k=1}^d \boldsymbol{\xi}_k.$$

Let $k \in \{1,...,d\}$. Since $\mathbf{y}$ is standard Gaussian, $u_k=y_k^2$ follows a $\chi^2(1)$ distribution, or equivalently a $\textup{Gamma}(1/2,1/2)$ distribution. Therefore, $u_k/2 \sim \textup{Gamma}(1/2,1)$. Moreover, note that $ \sqrt{u_k} \mathbf{w}_k=|y_k|\mathbf{w}_k=y_k \textup{sign}(y_k)\mathbf{w}_k\buildrel d \over =y_k\mathbf{w}_k$ since $|y_k|$ and $\textup{sign}(y_k)$ are independent and $\textup{sign}(y_k)\mathbf{w}_k\buildrel d\over =\mathbf{w}_k$. Therefore, according to the generalized Gauss-Laplace representation, we have
$$ \boldsymbol{\xi}_k \buildrel d \over = \sqrt{\frac{u_k}{2}} \sqrt{2}\mathbf{w}_k +  \frac{u_k}{2}2\boldsymbol{\mu} \sim  \textup{GAL}_p(2\mathbf{\Sigma},2\boldsymbol{\mu},1/2).$$

Since $\boldsymbol{\xi}_1,...,\boldsymbol{\xi}_d$ are i.i.d. and following a $ \textup{GAL}_p(2\mathbf{\Sigma},2\boldsymbol{\mu},1/2)$ distribution, we can use Proposition \ref{summation} to conclude that $$\mathbf{Wy} + ||\mathbf{y}||_2^2\boldsymbol{\mu} =\sum_{k=1}^d \boldsymbol{\xi}_k\sim  \textup{GAL}_p(2\mathbf{\Sigma}, 2\boldsymbol{\mu},d/2).$$
\end{proof}

In the symmetric case ($\boldsymbol{\mu}=0$), this result gives the distribution of the product between a Gaussian matrix with i.i.d. columns and a isotropic Gaussian vector.

\begin{corollary}
Let $\mathbf{W}$ be a $p \times d$ random matrix with i.i.d. columns following a $\mathcal{N}(0,\mathbf{\Sigma})$ distribution and let $\mathbf{y}\sim \mathcal{N}(0,\alpha\mathbf{I}_d)$ be a Gaussian vector independent from $\mathbf{W}$. Then \begin{equation}\mathbf{Wy} \sim  \textup{GAL}_p(2\alpha \mathbf{\Sigma}, 0,d/2).\end{equation}
Moreover, if $u$ is a standard Gamma variable with shape parameter $d/2$ and if $\mathbf{x} \sim \mathcal{N}(0,2\alpha \mathbf{\Sigma})$ is a Gaussian vector independent of $u$, then \begin{equation} \label{pca}
\mathbf{Wy} \buildrel d \over = \sqrt{u}\mathbf{x}.
\end{equation}
\end{corollary}

Less general versions of Theorem \ref{maintheo} have been proven in the past, dating back to the derivation of the inner product of two i.i.d. standard Gaussian vectors by \cite{wishart}. In particular, the unidimensional case ($p=1$) was recently proven by \cite{gaunt2014} in order to obtain bounds for the convergence rate of random sums involving Gaussian products. The multivariate symmetric isotropic case ($\boldsymbol{\mu}=0$ and $\mathbf{\Sigma}$ proportional to $\mathbf{I}_p$) was proven by \cite*{bouveyron2016} in order to derive the marginal likelihood of the noiseless probabilistic principal component analysis model of \cite{roweis1998}. While the proof of \cite*{bouveyron2016} relied on characteristic functions and the properties of Bessel functions, the proof that we presented here is closer in spirit to the one of \cite{gaunt2014}, based on representations of distributions.

\section{Perspectives}

The new characterization presented in this paper may notably prove useful in two contexts.

First, it indicates a new way of handling situations involving the product of a Gaussian matrix and a Gaussian vector. An important instance is the Bayesian factor analysis model \citep{lopes2004}, of which principal component analysis is a particular case. In this framework, the marginal distribution of the data, which is essential for model selection purposes, can be derived using representation \eqref{pca} together with the Gauss-Laplace representation \citep*{bouveyron2016,bouveyron2017}.

Moreover, our characterization offers a means to get around problems encountered when dealing with distributions related to the GAL distribution. For example, representation \eqref{mainres} might lead to alternative estimation strategies for some problems related to portfolio allocation \citep{mencia2009,breymann2013ghyp} or cluster analysis \citep*{mcnicholas2013,franczak2014}.

\section*{Acknowledgements}
I thank Charles Bouveyron, Pierre Latouche, Brendan Murphy and Christian Robert for fruitful advices and discussions. Part of this work was made during a visit to University College Dublin, funded by the Fondation Sciences Math\'ematiques de Paris (FSMP).

\bibliography{biblio}

\end{document}